\documentclass[preprint,12pt]{elsarticle}
\usepackage{amssymb}
\usepackage{amsthm}
\usepackage{amsmath}
\usepackage{epic}
\newtheorem{thm}{Theorem}[section]

\newtheorem{lem}[thm]{Lemma}

\newtheorem{definition}[thm]{Definition}

\journal{}

\begin{document}

\begin{frontmatter}

\title{E-cospectral hypergraphs and some hypergraphs determined by their spectra}
\author[label1]{Changjiang Bu}\ead{buchangjiang@hrbeu.edu.cn}
\author[label1,label2]{Jiang Zhou}\ead{zhoujiang04113112@163.com}
\author[label3]{Yimin Wei}\ead{ymwei@fudan.edu.cn}

\address[label1]{College of Science, Harbin Engineering University, Harbin 150001, PR China}
\address[label2]{College of Computer Science and Technology, Harbin Engineering University, Harbin 150001, PR China}
\address[label3]{School of Mathematical Sciences and Shanghai Key Laboratory of Contemporary Applied Mathematics, Fudan University, Shanghai, 200433, PR China}

\begin{abstract}
Two $k$-uniform hypergraphs are said to be cospectral (E-cospectral), if their adjacency tensors have the same characteristic polynomial (E-characteristic polynomial). A $k$-uniform hypergraph $H$ is said to be determined by its spectrum, if there is no other non-isomorphic $k$-uniform hypergraph cospectral with $H$. In this note, we give a method for constructing E-cospectral hypergraphs, which is similar with Godsil-McKay switching. Some hypergraphs are shown to be determined by their spectra.
\end{abstract}


\end{frontmatter}

\section{Introduction}
For a positive integer $n$, let $[n]=\{1,\ldots,n\}$. An order $k$ dimension $n$ tensor $\mathcal{A}=(a_{i_1\cdots i_k})\in\mathbb{C}^{n\times\cdots\times n}$ is a multidimensional array with $n^k$ entries, where $i_j\in[n]$, $j=1,\ldots,k$. $\mathcal{A}$ is called \textit{symmetric} if $a_{i_1i_2\cdots i_k}=a_{i_{\sigma(1)}i_{\sigma(2)}\cdots i_{\sigma(k)}}$ for any permutation $\sigma$ on $[k]$. We sometimes write $a_{i_1\cdots i_k}$ as $a_{i_1\alpha}$, where $\alpha=i_2\cdots i_k$. When $k=1$, $\mathcal{A}$ is a column vector of dimension $n$. When $k=2$, $\mathcal{A}$ is an $n\times n$ matrix. The \textit{unit tensor} of order $k\geqslant2$ and dimension $n$ is the tensor $\mathcal{I}_n=(\delta_{i_1i_2\cdots i_k})$ such that $\delta_{i_1i_2\cdots i_k}=1$ if $i_1=i_2=\cdots=i_k$, and $\delta_{i_1i_2\cdots i_k}=0$ otherwise. When $k=2$, $\mathcal{I}_n$ is the identity matrix $I_n$. Recently, Shao \cite{Shao-product} introduce the following product of tensors, which is a generalization of the matrix multiplication.
\begin{definition}\label{definition1.1}\textup{\cite{Shao-product}}
Let $\mathcal{A}$ and $\mathcal{B}$ be order $m\geqslant2$ and order $k\geqslant1$, dimension $n$ tensors, respectively. The product $\mathcal{A}\mathcal{B}$ is the following tensor $\mathcal{C}$ of order $(m-1)(k-1)+1$ and dimension $n$ with entries:
\begin{eqnarray*}
c_{i\alpha_1\ldots \alpha_{m-1}}=\sum_{i_2,\ldots,i_m\in[n]}a_{ii_2\ldots i_m}b_{i_2\alpha_1}\cdots b_{i_m\alpha_{m-1}},
\end{eqnarray*}
where $i\in[n]$, $\alpha_1,\ldots,\alpha_{m-1}\in[n]^{k-1}$.
\end{definition}

Let $\mathcal{A}$ be an order $m\geqslant2$ dimension $n$ tensor, and let $x=(x_1,\ldots,x_n)^\top$. From Definition \ref{definition1.1}, the product $\mathcal{A}x$ is a vector in $\mathbb{C}^n$ whose $i$-th component is (see Example 1.1 in \cite{Shao-product})
\begin{eqnarray*}
(\mathcal{A}x)_i=\sum_{i_2,\ldots,i_m\in[n]}a_{ii_2\cdots i_m}x_{i_2}\cdots x_{i_m}.
\end{eqnarray*}
In 2005, the concept of tensor eigenvalues was posed by Qi \cite{Qi05} and Lim \cite{Lim}. A number $\lambda\in\mathbb{C}$ is called an \textit{eigenvalue} of $\mathcal{A}$, if there exists a nonzero vector $x\in\mathbb{C}^n$ such that $\mathcal{A}x=\lambda x^{[m-1]}$, where $x^{[m-1]}=(x_1^{m-1},\ldots,x_n^{m-1})^\top$. The \textit{determinant} of $\mathcal{A}$, denoted by $\det(\mathcal{A})$, is the resultant of the system of polynomials $f_i(x_1,\ldots,x_n)=(\mathcal{A}x)_i$ ($i=1,\ldots,n$). The \textit{characteristic polynomial} of $\mathcal{A}$ is defined as $\Phi_{\mathcal{A}}(\lambda)=\det(\lambda\mathcal{I}_n-\mathcal{A})$, where $\mathcal{I}_n$ is the unit tensor of order $m$ and dimension $n$. It is known that eigenvalues of $\mathcal{A}$ are exactly roots of $\Phi_{\mathcal{A}}(\lambda)$ (see \cite{Shao-product}).

For an order $m\geqslant2$ dimension $n$ tensor $\mathcal{A}$, a number $\lambda\in\mathbb{C}$ is called an \textit{E-eigenvalue} of $\mathcal{A}$, if there exists a nonzero vector $x\in\mathbb{C}^n$ such that $\mathcal{A}x=\lambda x$ and $x^\top x=1$. In \cite{Qi07}, the \textit{E-characteristic polynomial} of $\mathcal{A}$ is defined as
\begin{eqnarray*}
\phi_\mathcal{A}(\lambda)=\begin{cases}{\rm Res}_x\left(\mathcal{A}x-\lambda(x^\top x)^{\frac{m-2}{2}}x\right)~~~~~~~m~\mbox{is even},\\{\rm Res}_{x,\beta}\begin{pmatrix}\mathcal{A}x-\lambda\beta^{m-2}x\\x^\top x-\beta^2\end{pmatrix}~~~~~~~~~~m~\mbox{is odd},\end{cases}
\end{eqnarray*}
where `Res' is the resultant of the system of polynomials. It is known that E-eigenvalues of $\mathcal{A}$ are roots of $\phi_\mathcal{A}(\lambda)$ (see \cite{Qi07}). If $m=2$, then $\phi_\mathcal{A}(\lambda)=\Phi_{\mathcal{A}}(\lambda)$ is just the characteristic polynomial of the square matrix $\mathcal{A}$.

A hypergraph $H$ is called $k$-\textit{uniform} if each edge of $H$ contains exactly $k$ distinct vertices. All hypergraphs in this note are uniform and simple. Let $K_n^k$ denote the complete $k$-uniform hypergraph with $n$ vertices, i.e., every $k$ distinct vertices of $K_n^k$ forms an edge. For a $k$-uniform hypergraph $H=(V(H),E(H))$, a hypergraph $G=(V(G),E(G))$ is a \textit{sub-hypergraph} of $H$, if $V(G)\subseteq V(H)$ and $E(G)\subseteq E(H)$. For any edge $u_1\cdots u_k\in E(H)$, we say that $u_k$ is a \textit{neighbor} of $\{u_1,\ldots,u_{k-1}\}$. The \textit{complement} of $H$ is a $k$-uniform hypergraph with vertex set $V(H)$ and edge set $E(K_{|V(H)|}^k)\backslash E(H)$. The \textit{adjacency tensor} of $H$, denoted by $\mathcal{A}_H$, is an order $k$ dimension $|V(H)|$ tensor with entries (see \cite{Cooper12})
\begin{eqnarray*}
a_{i_1i_2\cdots i_k}=\begin{cases}\frac{1}{(k-1)!}~~~~~~~\mbox{if}~i_1i_2\cdots i_k\in E(H),\\
0~~~~~~~~~~~~~\mbox{otherwise}.\end{cases}
\end{eqnarray*}
Clearly $\mathcal{A}_H$ is a symmetric tensor. We say that two $k$-uniform hypergraphs are \textit{cospectral} (\textit{E-cospectral}), if their adjacency tensors have the same characteristic polynomial (E-characteristic polynomial). A $k$-uniform hypergraph $H$ is said to be \textit{determined by its spectrum}, if there is no other non-isomorphic $k$-uniform hypergraph cospectral with $H$. We shall use ``DS" as an abbreviation for ``determined by its spectrum" in this note. Cospectral (E-cospectral) hypergraphs and DS hypergraphs are generalizations of cospectral graphs and DS graphs in the classic sense \cite{Dam03}.

Recently, the research on spectral theory of hypergraphs has attracted extensive attention [2,6-9,12,15-20]. In this note, we give a method for constructing E-cospectral hypergraphs. Some hypergraphs are shown to be DS.
\section{Preliminaries}
The following lemma can be obtained from equation (2.1) in \cite{Shao-product}.
\begin{lem}\label{lem1}
Let $\mathcal{A}=(a_{i_1\cdots i_m})$ be an order $m\geqslant2$ dimension $n$ tensor, and let $P=(p_{ij})$ be an $n\times n$ matrix. Then
\begin{eqnarray*}
(P\mathcal{A}P^\top)_{i_1\cdots i_m}=\sum_{j_1,\ldots,j_m\in[n]}a_{j_1\cdots j_m}p_{i_1j_1}p_{i_2j_2}\cdots p_{i_mj_m}.
\end{eqnarray*}
\end{lem}
We can obtain the following lemma from Lemma \ref{lem1}.
\begin{lem}\label{lem2}
Let $\mathcal{B}=P\mathcal{A}P^\top$, where $\mathcal{A}$ is a tensor of dimension $n$, $P$ is an $n\times n$ matrix. If $\mathcal{A}$ is symmetric, then $\mathcal{B}$ is symmetric.
\end{lem}
Let $\mathcal{B}=P\mathcal{A}P^\top$, where $\mathcal{A}$ is a tensor of dimension $n$, $P$ is an $n\times n$ real orthogonal matrix. In \cite{Shao-product}, Shao pointed out that $\mathcal{A},\mathcal{B}$ are orthogonally similar tensors defined by Qi \cite{Qi05}. Orthogonally similar tensors have the following property.
\begin{lem}\label{lem3}\textup{\cite{Li}}
Let $\mathcal{B}=P\mathcal{A}P^\top$, where $\mathcal{A}$ is a tensor of dimension $n$, $P$ is an $n\times n$ real orthogonal matrix. Then $\mathcal{A}$ and $\mathcal{B}$ have the same E-characteristic polynomial.
\end{lem}
A \textit{simplex} in a $k$-uniform hypergraph is a set of $k+1$ vertices where every set of $k$ vertices forms an edge (see [2, Definition 3.4]).
\begin{lem}\label{lem4}
Let $G$ and $H$ be cospectral $k$-uniform hypergraphs. Then $G$ and $H$ have the same number of vertices, edges and simplices.
\end{lem}
\begin{proof}
The degree of the characteristic polynomial of an order $k$ dimension $n$ tensors is $n(k-1)^{n-1}$ (see \cite{Qi05}). Since $\mathcal{A}_G$ and $\mathcal{A}_H$ are order $k$ tensors, $G$ and $H$ have the same number of vertices. From [2, Theorem 3.15] and [2, Theorem 3.17], we know that $G$ and $H$ have the same number of edges and simplices.
\end{proof}
\section{Main results}
Let $H=(V(H),E(H))$ be a $k$-uniform hypergraph with a partition $V(H)=V_1\cup V_2$, and $H$ satisfies the following conditions:

(a) For each edge $e\in E(H)$, $e$ contains at most one vertex in $V_1$.

(b) For any $k-1$ distinct vertices $u_1,\ldots,u_{k-1}\in V_2$, $\{u_1,\ldots,u_{k-1}\}$ has either $0,\frac{1}{2}|V_1|$ or $|V_1|$ neighbors in $V_1$.

Similar with GM switching \cite{GM switching,Haemers}, we construct a hypergraph E-cospectral with $H$ as follows.
\begin{thm}\label{thm1}
Let $H$ be a $k$-uniform hypergraph satisfies the conditions (a) and (b) described above. For any $\{u_1,\ldots,u_{k-1}\}\subseteq V_2$ which has $\frac{1}{2}|V_1|$ neighbors in $V_1$, by replacing these $\frac{1}{2}|V_1|$ neighbors with the other $\frac{1}{2}|V_1|$ vertices in $V_1$, we obtain a $k$-uniform hypergraph $G$ which is E-cospectral with $H$.
\end{thm}
\begin{proof}
Let $P=\begin{pmatrix}\frac{2}{n_1}J-I_{n_1}&0\\0&I_{n_2}\end{pmatrix}$, where $n_1=|V_1|,n_2=|V_2|$, $J$ is the $n_1\times n_1$ all-ones matrix, $\frac{2}{n_1}J-I_{n_1}$ and $I_{n_2}$ correspond to the vertex sets $V_1$ and $V_2$, respectively. Then $P=P^\top=P^{-1}$. Suppose that $\mathcal{A}_H=(a_{i_1i_2\cdots i_k})$, and let $\mathcal{B}=P\mathcal{A}_HP^\top$. By Lemma \ref{lem2}, $\mathcal{B}$ is symmetric. We need to show that $\mathcal{B}=\mathcal{A}_G$. By Lemma \ref{lem1}, we have
\begin{eqnarray}\label{Eq.1}
(\mathcal{B})_{i_1\cdots i_k}=\sum_{j_1,\ldots,j_k\in V(H)}a_{j_1\cdots j_k}p_{i_1j_1}p_{i_2j_2}\cdots p_{i_kj_k}.
\end{eqnarray}
Note that $P=\begin{pmatrix}\frac{2}{n_1}J-I_{n_1}&0\\0&I_{n_2}\end{pmatrix}$. From Eq. (\ref{Eq.1}), we have
\begin{eqnarray}\label{Eq.2}
(\mathcal{B})_{i_1\cdots i_k}=a_{i_1\cdots i_k}~~\mbox{if}~i_1,\ldots,i_k\in V_2.
\end{eqnarray}
Since $H$ satisfies the condition (a), we have $a_{j_1\cdots j_k}=0$ if $|\{j_1,\ldots,j_k\}\cap V_1|\geqslant2$. From Eq. (\ref{Eq.1}), we have
\begin{eqnarray}\label{Eq.3}
(\mathcal{B})_{i_1\cdots i_k}=a_{i_1\cdots i_k}=0~~\mbox{if}~|\{i_1,\ldots,i_k\}\cap V_1|\geqslant2.
\end{eqnarray}
Next we consider the case $|\{i_1,\ldots,i_k\}\cap V_1|=1$. Note that $\mathcal{B}$ is symmetric. Without loss of generality, suppose that $i_1\in V_1,i_2,\ldots,i_k\in V_2$. From Eq. (\ref{Eq.1}), we have
\begin{eqnarray}\label{Eq.4}
(\mathcal{B})_{i_1\cdots i_k}=\sum_{j_1\in V_1}a_{j_1i_2\cdots i_k}p_{i_1j_1}~(i_1\in V_1,i_2,\ldots,i_k\in V_2).
\end{eqnarray}
Since $H$ satisfies the condition (b), $S^{i_2\cdots i_k}=\{a_{j_1i_2\cdots i_k}|j_1\in V_1,a_{j_1i_2\cdots i_k}\neq0\}$ contains either $0,\frac{1}{2}|V_1|$ or $|V_1|$ elements for any given $i_2,\ldots,i_k\in V_2$. By computing the sum in (\ref{Eq.4}), we have
\begin{eqnarray}\label{Eq.5}
(\mathcal{B})_{i_1\cdots i_k}=a_{i_1\cdots i_k}=0~~\mbox{if}~i_1\in V_1,i_2,\ldots,i_k\in V_2,|S^{i_2\cdots i_k}|=0.
\end{eqnarray}
\begin{eqnarray}\label{Eq.6}
(\mathcal{B})_{i_1\cdots i_k}=a_{i_1\cdots i_k}=\frac{1}{(k-1)!}~~\mbox{if}~i_1\in V_1,i_2,\ldots,i_k\in V_2,|S^{i_2\cdots i_k}|=|V_1|.
\end{eqnarray}
\begin{eqnarray}\label{Eq.7}
(\mathcal{B})_{i_1\cdots i_k}=0~~\mbox{if}~i_1\in V_1,i_2,\ldots,i_k\in V_2,a_{i_1\cdots i_k}=\frac{1}{(k-1)!},|S^{i_2\cdots i_k}|=\frac{1}{2}|V_1|.
\end{eqnarray}
\begin{eqnarray}\label{Eq.8}
(\mathcal{B})_{i_1\cdots i_k}=\frac{1}{(k-1)!}~~\mbox{if}~i_1\in V_1,i_2,\ldots,i_k\in V_2,a_{i_1\cdots i_k}=0,|S^{i_2\cdots i_k}|=\frac{1}{2}|V_1|.
\end{eqnarray}
From Eqs. (\ref{Eq.2})(\ref{Eq.3}) and (\ref{Eq.5})-(\ref{Eq.8}), we have $\mathcal{B}=P\mathcal{A}_HP^\top=\mathcal{A}_G$. By Lemma \ref{lem3}, $G$ is E-cospectral with $H$.
\end{proof}
If two $k$-uniform hypergraphs $G$ and $H$ are isomorphic, then there exists a permutation matrix $P$ such that $\mathcal{A}_G=P\mathcal{A}_HP^\top$ (see \cite{Bu,Shao-product}). From Lemma \ref{lem3}, we know that two isomorphic $k$-uniform hypergraphs are E-cospectral. By using the method in Theorem \ref{thm1}, we give a class of non-isomorphic E-cospectral hypergraphs as follows.

\vspace{3mm}
\noindent
\textbf{Example.} Let $H$ be a $3$-uniform hypergraph whose vertex set and edge set are
\begin{eqnarray*}
V(H)&=&\{u_1,u_2,u_3,u_4,v_1,\ldots,v_n\}~(n\geqslant3),\\
E(H)&=&\{v_1v_2u_2,v_1v_2u_3,v_2v_3u_2,v_2v_3u_4,v_1v_3u_3,v_1v_3u_4\}\cup F,
\end{eqnarray*}
where each edge in $F$ contains three vertices in $\{v_1,\ldots,v_n\}$, and each vertex in $\{v_4,\ldots,v_n\}$ is contained in at least one edge in $F$ if $n\geqslant4$. Let $G$ be a $3$-uniform hypergraph whose vertex set and edge set are
\begin{eqnarray*}
V(G)=V(H),E(G)&=&\{v_1v_2u_1,v_1v_2u_4,v_2v_3u_1,v_2v_3u_3,v_1v_3u_1,v_1v_3u_2\}\cup F.
\end{eqnarray*}
The vertex set $V(H)$ has a partition $V(H)=V_1\cup V_2$ such that $H$ and $G$ satisfy the conditions in Theorem \ref{thm1}, where $V_1=\{u_1,u_2,u_3,u_4\}$, $V_2=\{v_1,\ldots,v_n\}$. Then $G$ and $H$ are E-cospectral. Moreover, $G$ and $H$ are non-isomorphic E-cospectral hypergraphs, because $H$ has an isolated vertex $u_1$ and $G$ has no isolated vertices.

\vspace{3mm}
Let $K_n^k-e$ denote the $k$-uniform hypergraph obtained from $K_n^k$ by deleting one edge. We can obtain the following result from Lemma \ref{lem4}.
\begin{thm}
The complete $k$-uniform hypergraph $K_n^k$, the hypergraph $K_n^k-e$ and their complements are DS. Any $k$-uniform sub-hypergraph of $K_{k+1}^k$ is DS. The disjoint union of $K_{k+1}^k$ and some isolated vertices is DS.
\end{thm}
If $G$ is a sub-hypergraph of a hypergraph $H$, then let $H\backslash G$ denote the hypergraph obtained from $H$ by deleting all edges of $G$.
\begin{thm}
The hypergraph $K_n^k\backslash G$ is DS, where $G$ is a $k$-uniform sub-hypergraph of $K_n^k$ such that all edges of $G$ share $k-1$ common vertices.
\end{thm}
\begin{proof}
Let $H$ be any $k$-uniform hypergraph cospectral with $K_n^k\backslash G$. Suppose that $G$ has $r$ edges. By Lemma \ref{lem4}, $H$ can be obtained from $K_n^k$ by deleting $r$ edges. Deleting $r$ edges from $K_n^k$ destroys at least $\sum_{i=0}^{r-1}(n-k-i)$ simplices, with equality if and only if all deleted edges share $k-1$ common vertices. Lemma \ref{lem4} implies that $H=K_n^k\backslash G$.
\end{proof}

\vspace{3mm}
\noindent
\textbf{Acknowledgements.}

\vspace{3mm}

This work is supported by the National Natural Science Foundation of China (No. 11371109 and No. 11271084), and the Fundamental Research Funds for the Central Universities.


\begin{thebibliography}{}
\bibitem{Bu}C. Bu, X. Zhang, J. Zhou, W. Wang, Y. Wei, The inverse, rank and product of tensors, Linear Algebra Appl. 446 (2014) 269--280.
\bibitem{Cooper12}J. Cooper, A. Dutle, Spectra of uniform hypergraphs, Linear Algebra Appl. 436 (2012) 3268--3292.
\bibitem{Dam03}E.R. van Dam, W.H. Haemers, Which graphs are determined by their spectrum?, Linear Algebra Appl. 373 (2003) 241--272.
\bibitem{GM switching}C.D. Godsil, B.D. McKay, Constructing cospectral graphs, Aequationes Math. 25 (1982) 257--268.
\bibitem{Haemers}W.H. Haemers, E. Spence, Enumeration of cospectral graphs, European. J. Combin. 25 (2004) 199--211.
\bibitem{HuQi1}S. Hu, L. Qi, Algebraic connectivity of an even uniform hypergraph, J. Comb. Optim. 24 (2012) 564--579.
\bibitem{HuQi2}S. Hu, L. Qi, The Laplacian of a uniform hypergraph, J. Comb. Optim., DOI:10.1007/s10878-013-9596-x.
\bibitem{HuQi-DAM}S. Hu, L. Qi, The eigenvectors associated with the zero eigenvalues of the Laplacian and signless Laplacian tensors of a uniform hypergraph, Discrete Appl. Math. 169 (2014) 140--151.
\bibitem{HuQiShao}S. Hu, L. Qi, J.Y. Shao, Cored hypergraphs, power hypergraphs and their Laplacian H-eigenvalues, Linear Algebra Appl. 439 (2013) 2980--2998.
\bibitem{Li}A.M. Li, L. Qi, B. Zhang, E-characteristic polynomials of tensors, Commun. Math. Sci. 11 (2013) 33--53.
\bibitem{Lim}L.H. Lim, Singular values and eigenvalues of tensors: a variational approach, in: Proceedings of the IEEE International Workshop on Computational Advances in Multi-Sensor Adaptive Processing, vol.1, 2005, pp. 129--132.
\bibitem{Pearson}K. Pearson, T. Zhang, On spectral hypergraph theory of the adjacency tensor, Graphs Combin., DOI:10.1007/s00373-013-1340-x.
\bibitem{Qi05}L. Qi, Eigenvalues of a real supersymmetric tensor, J. Symbolic Comput. 40 (2005) 1302--1324.
\bibitem{Qi07}L. Qi, Eigenvalues and invariants of tensors, J. Math. Anal. Appl. 325 (2007) 1363--1377.
\bibitem{Qi-Laplaican}L. Qi, H$^+$-eigenvalues of Laplacian and signless Laplacian tensors, Commun. Math. Sci. 12 (2014) 1045--1064.
\bibitem{QiShaoWang}L. Qi, J.Y. Shao, Q. Wang, Regular uniform hypergraphs, s-cycles, s-paths and their largest Laplacian H-eigenvalues, Linear Algebra Appl. 443 (2014) 215--227.
\bibitem{Shao-product}J.Y. Shao, A general product of tensors with applications, Linear Algebra Appl. 439 (2013) 2350--2366.
\bibitem{ShaoQiHu}J.Y. Shao, L. Qi, S. Hu, Some new trace formulas of tensors with applications in spectral hypergraph theory, to appear in: Linear and Multilinear Algebra.
\bibitem{ShaoShanWu}J.Y. Shao, H.Y. Shan, B. Wu, Some spectral properties and characterizations of connected odd-bipartite uniform hypergraphs, arXiv:1403.4845v1.
\bibitem{Xie-LAA}J. Xie, A. Chang, On the Z-eigenvalues of the adjacency tensors for uniform hypergraphs, Linear Algebra Appl. 439 (2013) 2195--2204.
\end{thebibliography}
\end{document}